\documentclass[11pt]{amsart}

\usepackage{amsmath,amssymb,amsthm,mathtools}
\usepackage{enumitem}
\usepackage{hyperref}
\usepackage[margin=1.15in]{geometry}

\newtheorem{theorem}{Theorem}[section]
\newtheorem{lemma}[theorem]{Lemma}

\newcommand{\E}{\mathbb E}

\title{Sharp Metric $X_p$ Inequalities via Martingales}

\author{Nika Areshidze}
\address{
Department of Mathematics,
University of California, Irvine,
CA 92697, USA
}
\email{nareshid@uci.edu}



\begin{document}

\begin{abstract}
We prove a sharp higher-order Rademacher-chaos inequality on the Hamming cube using random-reveal martingales. As a consequence, we obtain the sharp metric $X_p$ inequality for $L^p$, resolving an open question posed by Naor.
\end{abstract}

\maketitle

\section{Introduction}

Metric $X_p$ inequalities were introduced by Naor and Schechtman
\cite{NaorSchechtman} in their study of the nonlinear geometry of
$L^p$ spaces. A related first-order inequality for Rademacher chaos
was originally proved by Johnson, Maurey, Schechtman and Tzafriri
\cite{JMST}. Johnson, Schechtman and Zinn \cite{JSZ} subsequently
obtained its sharp form, with optimal dependence of order $p/\log p$.

In \cite{Naor}, Naor extended this first-order inequality to
Rademacher chaos of arbitrary degree. For a mean-zero function on the
Hamming cube, the resulting estimate controls averages of conditional
expectations over subsets of coordinates in terms of the discrete
derivatives of the function and its $L^p$ norm. This higher-order
inequality was a central ingredient in Naor's proof of the metric
$X_p$ inequality for $L^p$ with the sharp scaling parameter.

The dependence on $p$ obtained by this argument, however, was not
optimal. Naor's proof yields an upper bound of order
$p^4/\log p$, whereas the first-order inequality shows that order
$p/\log p$ is necessary. In \cite[Remark~5]{Naor}, Naor observed that,
although a more careful implementation of his approach might somewhat
improve the dependence on $p$, obtaining the sharp order appeared to
require a new idea, and he left the sharp dependence on $p$ as an open problem.

At the level of the underlying Rademacher-chaos
inequality, the corresponding question was left open in
\cite[Remark~7]{Naor}.

In this paper we obtain the sharp $p/\log p$ dependence in the
higher-order Rademacher-chaos inequality, thereby closing the gap
between the first-order estimate and Naor's higher-order bound. In
particular, we prove the following.

\begin{theorem}[Sharp Rademacher-chaos inequality]\label{thm:chaos}
For every $p\ge2$, every $n\in\mathbb N$, every $k\in[n]$, and every
mean-zero function $h:\Omega_n\to\mathbb R$,
$$
\left(
\frac{1}{\binom{n}{k}}
\sum_{\substack{S\subseteq[n]\\ |S|=k}}
\|E_{[n]\setminus S}h\|_p^p
\right)^{1/p}
\lesssim
\frac{p}{\log p}
\left(
\frac{k}{n}\sum_{j=1}^n\|\partial_j h\|_p^p
+
\left(\frac{k}{n}\right)^{p/2}\|h\|_p^p
\right)^{1/p}.
$$
Moreover, the order $p/\log p$ is sharp up to a universal
multiplicative constant.
\end{theorem}

The proof takes a different route from the one in \cite{Naor}. We
reveal the coordinates of the Hamming cube in a uniformly random
order, thereby realizing the averaging over $k$-element subsets as a
random-reveal martingale. Hitczenko's sharp martingale Rosenthal
inequality \cite{Hitczenko} can then be applied directly, producing
the optimal $p/\log p$ dependence.

Finally, combining Theorem~\ref{thm:chaos} with Naor's reduction from
the Rademacher-chaos inequality to the metric $X_p$ inequality yields
the corresponding sharp metric estimate for $L^p$. In particular, we
obtain the metric $X_p$ inequality with the sharp scaling parameter
and optimal dependence $p/\log p$, resolving the question raised in
\cite[Remark~5]{Naor}. The precise statement appears in
Theorem~\ref{thm:metric}.

The remainder of the paper is organized as follows. Section~2 contains
the notation and preliminaries. In Section~3 we prove
Theorem~\ref{thm:chaos}. Section~4 applies Naor's reduction to deduce
the sharp metric $X_p$ inequality.

\section{Preliminaries}

\subsection{The Hamming cube and conditional expectations}

For $n\in\mathbb N$, write $[n]=\{1,\ldots,n\}$
and let $\Omega_n=\{-1,1\}^n$ be equipped with the uniform probability
measure. Thus, for every $f:\Omega_n\to\mathbb R$,
$$
\mathbb Ef
=
\frac{1}{2^n}\sum_{x\in\Omega_n}f(x),
$$
and, for $1\le p<\infty$,
$$
\|f\|_p
=
\left(
\frac{1}{2^n}\sum_{x\in\Omega_n}|f(x)|^p
\right)^{1/p}.
$$

Throughout the paper, $A\lesssim B$ means that there exists a universal
constant $C>0$ such that $A\le CB$.

For $x=(x_1,\ldots,x_n)\in\Omega_n$ and $j\in[n]$, let $x^{(j)}$
denote the point obtained from $x$ by flipping its $j$th coordinate:
$$
x^{(j)}
=
(x_1,\ldots,x_{j-1},-x_j,x_{j+1},\ldots,x_n).
$$
For $f:\Omega_n\to\mathbb R$, define
$$
\partial_j f(x)=f(x)-f(x^{(j)}),
$$
and
$$
D_jf(x)
=
\frac12\partial_jf(x)
=
\frac{f(x)-f(x^{(j)})}{2}.
$$

Let $e_1,\ldots,e_n$ denote the standard basis of $\mathbb R^n$. For
$\varepsilon=(\varepsilon_1,\ldots,\varepsilon_n)\in\Omega_n$ and
$S\subseteq[n]$, write
$$
\varepsilon_S=\sum_{j\in S}\varepsilon_j e_j.
$$
For $h:\Omega_n\to\mathbb R$, define
$$
E_Sh(\varepsilon)
=
\frac{1}{2^n}
\sum_{\delta\in\Omega_n}
h\bigl(\delta_S+\varepsilon_{[n]\setminus S}\bigr).
$$
Thus $E_Sh$ is obtained by averaging over the coordinates in $S$ while
keeping the coordinates in $[n]\setminus S$ fixed.

Equivalently, if $X=(X_1,\ldots,X_n)$ is uniformly distributed on
$\Omega_n$, then
$$
E_Sh(X)
=
\mathbb E\!\left[
h(X)\,\middle|\,
\sigma(X_j:j\in[n]\setminus S)
\right].
$$

We will repeatedly use the fact that conditional expectation is an
$L^q$ contraction for every $q\ge1$:
$$
\|\mathbb E[f\mid\mathcal G]\|_q\le\|f\|_q.
$$
We will also use conditional Jensen's inequality in the form
$$
|\mathbb E[f\mid\mathcal G]|^q
\le
\mathbb E[|f|^q\mid\mathcal G],
\qquad q\ge1.
$$
\subsection{Martingales and the sharp Rosenthal inequality}

Let $(M_r)_{r=0}^N$ be a real-valued martingale with respect to a
filtration $(\mathcal F_r)_{r=0}^N$. We denote its martingale differences
by
$$
d_r=M_r-M_{r-1},
\qquad 1\le r\le N.
$$
The corresponding predictable square function is
$$
s_N
=
\left(
\sum_{r=1}^N
\mathbb E[d_r^2\mid\mathcal F_{r-1}]
\right)^{1/2}.
$$

The main martingale inequality used in the proof is the sharp martingale
version of Rosenthal's inequality due to Hitczenko \cite{Hitczenko}.

\begin{theorem}[Sharp martingale Rosenthal inequality]\label{thm:BR}
There exists a universal constant $C>0$ such that for every $p\ge2$
and every finite real-valued martingale $(M_r)_{r=0}^N$ with $M_0=0$,
with martingale differences $d_r=M_r-M_{r-1}$, one has
$$
\left\|\max_{0\le r\le N}|M_r|\right\|_p
\le
C\frac{p}{\operatorname{Log}p}
\left[
\left\|
\left(
\sum_{r=1}^N
\mathbb E[d_r^2\mid\mathcal F_{r-1}]
\right)^{1/2}
\right\|_p
+
\left\|\max_{1\le r\le N}|d_r|\right\|_p
\right],
$$
where
$$
\operatorname{Log}p=\max\{1,\log p\}.
$$
Moreover, the order $p/\operatorname{Log}p$ is optimal.
\end{theorem}

Since
$$
|M_k|\le\max_{0\le r\le k}|M_r|,
$$
we will apply Theorem~\ref{thm:BR} to the martingale stopped at time $k$.

\subsection{Metric $X_p$ spaces}

For $m,n\in\mathbb N$, write $\mathbb Z_{2m}^n=(\mathbb Z/(2m\mathbb Z))^n,$
equipped with normalized counting measure. 

Fix $p\in(0,\infty)$. Following Naor and Schechtman
\cite{NaorSchechtman}, a metric space $(X,d_X)$ is said to be an
$X_p$ metric space if there exists $\mathfrak X\in(0,\infty)$ such
that for every $n\in\mathbb N$ and every $k\in[n]$ there exists
$m\in\mathbb N$ for which every mapping
$$
f:\mathbb Z_{2m}^n\to X
$$
satisfies
$$
\begin{aligned}
&\left(
\frac1{\binom nk}
\sum_{\substack{S\subseteq[n]\\ |S|=k}}
\mathbb E_{x,\varepsilon}
\left[
d_X\bigl(f(x+m\varepsilon_S),f(x)\bigr)^p
\right]
\right)^{1/p}
\\
&\qquad\le
\mathfrak X\,m
\left(
\frac{k}{n}\sum_{j=1}^n
\mathbb E_{x}
\left[
d_X\bigl(f(x+e_j),f(x)\bigr)^p
\right]
+
\left(\frac{k}{n}\right)^{p/2}
\mathbb E_{x,\varepsilon}
\left[
d_X\bigl(f(x+\varepsilon),f(x)\bigr)^p
\right]
\right)^{1/p},
\end{aligned}
$$
where $(x,\varepsilon)$ is uniformly distributed on
$\mathbb Z_{2m}^n\times\{-1,1\}^n$.

\section{Proof of Theorem~\ref{thm:chaos}}

Let $X=(X_1,\ldots,X_n)$ be uniformly distributed on $\Omega_n$. Thus,
$X_1,\ldots,X_n$ are independent Rademacher random variables:
$$
\mathbb P(X_j=1)=\mathbb P(X_j=-1)=\frac12.
$$

Independently of $X$, let
$$
\pi=(\pi_1,\ldots,\pi_n)
$$
be a uniformly distributed random permutation of $[n]$. For
$0\le r\le n$, put
$$
S_r=\{\pi_1,\ldots,\pi_r\},
\qquad S_0=\emptyset.
$$
Let $\mathcal F_0$ be the trivial $\sigma$-algebra and, for $1\le r\le n$,
$$
\mathcal F_r
=
\sigma(\pi_1,\ldots,\pi_r,X_{\pi_1},\ldots,X_{\pi_r}).
$$
Set
$$
M_r=\mathbb E[h(X)\mid\mathcal F_r].
$$
Then $(M_r)_{r=0}^n$ is a martingale with respect to
$(\mathcal F_r)_{r=0}^n$. Since $h$ has mean zero,
$$
M_0=\mathbb Eh=0.
$$

Fix $k\in[n]$. By the definition of the filtration,
$$
M_k=E_{[n]\setminus S_k}h(X).
$$
Since $S_k$ is uniformly distributed over all $k$-element subsets of
$[n]$, we have
$$
\begin{aligned}
\|M_k\|_{L^p(X,\pi)}^p
&=
\mathbb E_{X,\pi}|E_{[n]\setminus S_k}h(X)|^p\\
&=
\mathbb E_\pi\|E_{[n]\setminus S_k}h\|_p^p\\
&=
\frac{1}{\binom{n}{k}}
\sum_{\substack{S\subseteq[n]\\ |S|=k}}
\|E_{[n]\setminus S}h\|_p^p.
\end{aligned}
$$

Let
$$
d_r=M_r-M_{r-1}.
$$
Applying Theorem~\ref{thm:BR} to the martingale
$(M_r)_{r=0}^k$ gives
$$
\begin{aligned}
\|M_k\|_{L^p(X,\pi)}
&\le
\left\|\max_{0\le r\le k}|M_r|\right\|_{L^p(X,\pi)}\\
&\le
C\frac{p}{\operatorname{Log}p}
\left[
\left\|
\left(
\sum_{r=1}^k
\mathbb E[d_r^2\mid\mathcal F_{r-1}]
\right)^{1/2}
\right\|_{L^p(X,\pi)}
+
\left\|\max_{1\le r\le k}|d_r|\right\|_{L^p(X,\pi)}
\right].
\end{aligned}
$$
Moreover,
$$
\left\|\max_{1\le r\le k}|d_r|\right\|_{L^p(X,\pi)}^p
\le
\sum_{r=1}^k\|d_r\|_{L^p(X,\pi)}^p.
$$
Thus it is enough to prove the two norm estimates
$$
\left\|
\left(
\sum_{r=1}^k
\mathbb E[d_r^2\mid\mathcal F_{r-1}]
\right)^{1/2}
\right\|_{L^p(X,\pi)}
\lesssim
\sqrt{\frac{k}{n}}\,\|h\|_p,
$$
and
$$
\left(
\sum_{r=1}^k\|d_r\|_{L^p(X,\pi)}^p
\right)^{1/p}
\lesssim
\left(
\frac{k}{n}\sum_{j=1}^n\|\partial_jh\|_p^p
\right)^{1/p}.
$$

We first consider the case $k\le n/2$. The case $k>n/2$ will be
treated separately at the end of the proof.

Fix $1\le r\le k$. Conditional on $\mathcal F_{r-1}$, the unrevealed
coordinates
$$
\{X_j:j\in[n]\setminus S_{r-1}\}
$$
are independent Rademacher random variables. Hence, for every
$j\in[n]\setminus S_{r-1}$,
$$
\mathbb E[h(X)\mid\mathcal F_{r-1},X_j]
=
\mathbb E[h(X)\mid\mathcal F_{r-1}]
+
X_j\mathbb E[X_jh(X)\mid\mathcal F_{r-1}].
$$
For fixed $r$, set
$$
a_j=\mathbb E[X_jh(X)\mid\mathcal F_{r-1}],
\qquad j\in[n]\setminus S_{r-1}.
$$
On the event $\{\pi_r=j\}$, conditioning additionally on $\pi_r=j$
does not give any further information about $X$, and hence
$$
M_r=M_{r-1}+X_ja_j.
$$
Therefore, on $\{\pi_r=j\}$,
$$
d_r=X_ja_j.
$$
Consequently,
$$
d_r
=
\sum_{j\in[n]\setminus S_{r-1}}
X_ja_j\mathbf 1_{\{\pi_r=j\}}.
$$
Since the events $\{\pi_r=j\}$ are pairwise disjoint,
$$
d_r^2
=
\sum_{j\in[n]\setminus S_{r-1}}
a_j^2\mathbf 1_{\{\pi_r=j\}}.
$$
Since $a_j$ is $\mathcal F_{r-1}$-measurable,
$$
\mathbb E[d_r^2\mid\mathcal F_{r-1}]
=
\sum_{j\in[n]\setminus S_{r-1}}
a_j^2
\mathbb P(\pi_r=j\mid\mathcal F_{r-1}).
$$
Given $\mathcal F_{r-1}$, the next index $\pi_r$ is uniformly
distributed among the $n-r+1$ unrevealed indices. Therefore
$$
\mathbb P(\pi_r=j\mid\mathcal F_{r-1})
=
\frac{1}{n-r+1},
\qquad j\in[n]\setminus S_{r-1},
$$
and hence
$$
\mathbb E[d_r^2\mid\mathcal F_{r-1}]
=
\frac{1}{n-r+1}
\sum_{j\in[n]\setminus S_{r-1}}a_j^2.
$$

Fix an atom $A$ of $\mathcal F_{r-1}$. On $A$, the family
$$
\{X_j:j\in[n]\setminus S_{r-1}\}
$$
forms an orthonormal system in
$L^2(A,\mathbb P(\,\cdot\,\mid A))$, since for
$i,j\in[n]\setminus S_{r-1}$,
$$
\mathbb E[X_iX_j\mid A]=\delta_{ij}.
$$
Moreover, on $A$,
$$
a_j
=
\mathbb E[X_jh(X)\mid A]
=
\langle h,X_j\rangle_{L^2(A,\mathbb P(\cdot\mid A))}.
$$
Hence, by Bessel's inequality,
$$
\sum_{j\in[n]\setminus S_{r-1}}
|\mathbb E[X_jh(X)\mid A]|^2
\le
\mathbb E[h(X)^2\mid A].
$$
Since $A$ was arbitrary,
$$
\sum_{j\in[n]\setminus S_{r-1}}a_j^2
\le
\mathbb E[h(X)^2\mid\mathcal F_{r-1}].
$$
Combining the last two estimates gives
$$
\mathbb E[d_r^2\mid\mathcal F_{r-1}]
\le
\frac{1}{n-r+1}
\mathbb E[h(X)^2\mid\mathcal F_{r-1}].
$$

Define the predictable square function
$$
s_k
=
\left(
\sum_{r=1}^k
\mathbb E[d_r^2\mid\mathcal F_{r-1}]
\right)^{1/2}.
$$
Then
$$
\|s_k\|_p^2
=
\left\|
\sum_{r=1}^k
\mathbb E[d_r^2\mid\mathcal F_{r-1}]
\right\|_{p/2}.
$$
Since $p\ge2$, the triangle inequality in $L^{p/2}$ gives
$$
\begin{aligned}
\|s_k\|_p^2
&\le
\sum_{r=1}^k
\left\|
\mathbb E[d_r^2\mid\mathcal F_{r-1}]
\right\|_{p/2}\\
&\le
\sum_{r=1}^k
\frac{1}{n-r+1}
\left\|
\mathbb E[h^2\mid\mathcal F_{r-1}]
\right\|_{p/2}.
\end{aligned}
$$
Since conditional expectation is a contraction on $L^{p/2}$,
$$
\left\|
\mathbb E[h^2\mid\mathcal F_{r-1}]
\right\|_{p/2}
\le
\|h^2\|_{p/2}
=
\|h\|_p^2.
$$
Thus
$$
\|s_k\|_p^2
\le
\left(
\sum_{r=1}^k\frac{1}{n-r+1}
\right)\|h\|_p^2.
$$
Since $k\le n/2$, for $1\le r\le k$,
$$
n-r+1\ge n-k+1\ge\frac n2,
$$
and therefore
$$
\sum_{r=1}^k\frac{1}{n-r+1}
\le
\frac{2k}{n}.
$$
Consequently,
$$
\|s_k\|_p
\le
\sqrt{\frac{2k}{n}}\|h\|_p.
$$

We now estimate the martingale differences. Fix $1\le r\le k$ and
$j\in[n]\setminus S_{r-1}$. On the event $\{\pi_r=j\}$, we have already
shown that
$$
\mathbb E[h(X)\mid\mathcal F_r]
=
\mathbb E[h(X)\mid\mathcal F_{r-1}]
+
X_j\mathbb E[X_jh(X)\mid\mathcal F_{r-1}].
$$
On the event $\{\pi_r=j\}$, the sigma-algebra $\mathcal F_r$ is
obtained from $\mathcal F_{r-1}$ by adjoining the revealed coordinate
$X_j$. Flipping this coordinate sends $X_j$ to $-X_j$, while the
conditional averaging over all remaining unrevealed coordinates is
unchanged. Therefore
$$
\mathbb E[h(X^{(j)})\mid\mathcal F_r]
=
\mathbb E[h(X)\mid\mathcal F_{r-1}]
-
X_j\mathbb E[X_jh(X)\mid\mathcal F_{r-1}].
$$
Hence
$$
\mathbb E[D_jh(X)\mid\mathcal F_r]
=
X_j\mathbb E[X_jh(X)\mid\mathcal F_{r-1}]
=
d_r
$$
on $\{\pi_r=j\}$. Thus
$$
d_r
=
\mathbb E[D_{\pi_r}h(X)\mid\mathcal F_r].
$$
By conditional Jensen's inequality,
$$
|d_r|^p
\le
\mathbb E[|D_{\pi_r}h(X)|^p\mid\mathcal F_r].
$$
Taking expectations and using that, marginally, $\pi_r$ is uniformly
distributed on $[n]$ and is independent of $X$, we obtain
$$
\mathbb E_{X,\pi}|d_r|^p
\le
\frac1n\sum_{j=1}^n\mathbb E_X|D_jh(X)|^p.
$$
Summing over $r=1,\ldots,k$ gives
$$
\sum_{r=1}^k\|d_r\|_{L^p(X,\pi)}^p
\le
\frac{k}{n}\sum_{j=1}^n\|D_jh\|_p^p.
$$
Since $2D_jh=\partial_jh$,
$$
\|D_jh\|_p^p
=
2^{-p}\|\partial_jh\|_p^p.
$$
Consequently,
$$
\sum_{r=1}^k\|d_r\|_{L^p(X,\pi)}^p
\le
\frac{1}{2^p}\frac{k}{n}
\sum_{j=1}^n\|\partial_jh\|_p^p.
$$
Taking $p$th roots,
$$
\left(
\sum_{r=1}^k\|d_r\|_{L^p(X,\pi)}^p
\right)^{1/p}
\le
\frac12
\left(
\frac{k}{n}\sum_{j=1}^n\|\partial_jh\|_p^p
\right)^{1/p}.
$$
Together with the square-function estimate, this proves the desired
inequality when $k\le n/2$.

Now assume $k>n/2$. For every $S\subseteq[n]$, the operator
$E_{[n]\setminus S}$ is a conditional expectation and therefore an
$L^p$ contraction. Hence
$$
\left(
\frac{1}{\binom{n}{k}}
\sum_{\substack{S\subseteq[n]\\ |S|=k}}
\|E_{[n]\setminus S}h\|_p^p
\right)^{1/p}
\le
\|h\|_p.
$$
Since $k>n/2$,
$$
\|h\|_p
\le
\sqrt{2}\sqrt{\frac{k}{n}}\|h\|_p.
$$
Also,
$$
\sqrt{\frac{k}{n}}\|h\|_p
\le
\left(
\frac{k}{n}\sum_{j=1}^n\|\partial_jh\|_p^p
+
\left(\frac{k}{n}\right)^{p/2}\|h\|_p^p
\right)^{1/p}.
$$
Therefore
$$
\left(
\frac{1}{\binom{n}{k}}
\sum_{\substack{S\subseteq[n]\\ |S|=k}}
\|E_{[n]\setminus S}h\|_p^p
\right)^{1/p}
\le
\sqrt2
\left(
\frac{k}{n}\sum_{j=1}^n\|\partial_jh\|_p^p
+
\left(\frac{k}{n}\right)^{p/2}\|h\|_p^p
\right)^{1/p}.
$$
Since $p/\log p$ is bounded below by a positive universal constant for
$p\ge2$, the factor $\sqrt2$ can be absorbed into the implicit constant.
This proves the estimate for all $k\in[n]$.

The factor $p/\log p$ is optimal up to a universal multiplicative
constant already in the first-order case, by Johnson, Schechtman and
Zinn \cite{JSZ}. Hence no smaller asymptotic order can hold uniformly
for all $n$ and $k$. This completes the proof.

\section{Deduction of the metric $X_p$ inequality}

We now deduce the metric $X_p$ inequality from
Theorem~\ref{thm:chaos}. The argument follows the reduction introduced
by Naor in \cite{Naor}. We include the details in order to keep track of
the dependence on $p$.

We first work in the scalar-valued setting and assume that
$$
f:\mathbb Z_{8m}^n\to\mathbb R.
$$
For $S\subseteq[n]$, define the averaging operator
$$
T_Sf(x)
=
\frac1{2^n}
\sum_{\delta\in\{-1,1\}^n}
f(x+2\delta_S),
\qquad x\in\mathbb Z_{8m}^n.
$$

We first record the following smoothing estimate.

\begin{lemma}[Naor~\cite{Naor}, Lemma~8]\label{lem:smoothing}
Let $p\in[1,\infty)$. For every $S\subseteq[n]$,
$$
\left(
\mathbb E_x
|f(x)-T_Sf(x)|^p
\right)^{1/p}
\le
2
\left(
\mathbb E_{x,\varepsilon}
|f(x+\varepsilon)-f(x)|^p
\right)^{1/p}.
$$
\end{lemma}

We now follow Naor's reduction. The only point at which we modify the
argument is the application of the Rademacher-chaos inequality, where
we use Theorem~\ref{thm:chaos}.

\begin{lemma}\label{lem:metric-main}
Suppose that $p\ge2$ and $k\in[n]$. Then
$$
\begin{aligned}
&\left(
\frac1{\binom nk}
\sum_{\substack{S\subseteq[n]\\|S|=k}}
\mathbb E_{x,\varepsilon}
\left|
T_{[n]\setminus S}f(x+4m\varepsilon_S)
-
T_{[n]\setminus S}f(x)
\right|^p
\right)^{1/p}
\\
&\qquad\lesssim
\frac{p}{\log p}\,m
\left(
\frac{k}{n}
\sum_{j=1}^n
\mathbb E_x
|f(x+e_j)-f(x)|^p
+
\left(\frac{k}{n}\right)^{p/2}
\mathbb E_{x,\varepsilon}
|f(x+\varepsilon)-f(x)|^p
\right)^{1/p}.
\end{aligned}
$$
\end{lemma}

\begin{proof}
Fix $S\subseteq[n]$. By telescoping and the triangle inequality in $L^p$,
$$
\begin{aligned}
&\left(
\mathbb E_{x,\varepsilon}
\left|
T_{[n]\setminus S}f(x+4m\varepsilon_S)
-
T_{[n]\setminus S}f(x)
\right|^p
\right)^{1/p}
\\
&\qquad\le
m
\left(
\mathbb E_{x,\varepsilon}
\left|
T_{[n]\setminus S}f(x+2\varepsilon_S)
-
T_{[n]\setminus S}f(x-2\varepsilon_S)
\right|^p
\right)^{1/p}.
\end{aligned}
$$
Here we used translation invariance of the uniform measure on
$\mathbb Z_{8m}^n$ in each term of the telescoping sum.

For every $x\in\mathbb Z_{8m}^n$, define
$$
h_x(\varepsilon)
=
f(x+2\varepsilon)-f(x-2\varepsilon),
\qquad
\varepsilon\in\{-1,1\}^n.
$$
Notice that $h_x$ is odd, and therefore
$$
\mathbb E_{\varepsilon}h_x(\varepsilon)=0.
$$
Moreover,
$$
T_{[n]\setminus S}f(x+2\varepsilon_S)
-
T_{[n]\setminus S}f(x-2\varepsilon_S)
=
E_{[n]\setminus S}h_x(\varepsilon).
$$
Consequently,
$$
\begin{aligned}
&\frac1{\binom nk}
\sum_{\substack{S\subseteq[n]\\|S|=k}}
\mathbb E_{x,\varepsilon}
\left|
T_{[n]\setminus S}f(x+4m\varepsilon_S)
-
T_{[n]\setminus S}f(x)
\right|^p
\\
&\qquad\le
m^p
\mathbb E_x
\left[
\frac1{\binom nk}
\sum_{\substack{S\subseteq[n]\\|S|=k}}
\|E_{[n]\setminus S}h_x\|_p^p
\right].
\end{aligned}
$$

Let $C_0>0$ be a universal constant for which
Theorem~\ref{thm:chaos} holds. Applying that theorem to $h_x$,
separately for every $x\in\mathbb Z_{8m}^n$, and then averaging in
$x$, yields
$$
\begin{aligned}
&\frac1{\binom nk}
\sum_{\substack{S\subseteq[n]\\|S|=k}}
\mathbb E_{x,\varepsilon}
\left|
T_{[n]\setminus S}f(x+4m\varepsilon_S)
-
T_{[n]\setminus S}f(x)
\right|^p
\\
&\qquad\le
\left(C_0\frac{p}{\log p}\right)^p m^p
\left[
\frac{k}{n}
\sum_{j=1}^n
\mathbb E_x\|\partial_jh_x\|_p^p
+
\left(\frac{k}{n}\right)^{p/2}
\mathbb E_x\|h_x\|_p^p
\right].
\end{aligned}
$$

It remains to estimate the two terms on the right. For $j\in[n]$,
$$
\begin{aligned}
\partial_jh_x(\varepsilon)
&=
f(x+2\varepsilon)-f(x-2\varepsilon)
\\
&\quad
-f(x+2\varepsilon-4\varepsilon_je_j)
+f(x-2\varepsilon+4\varepsilon_je_j).
\end{aligned}
$$
Hence
$$
\begin{aligned}
|\partial_jh_x(\varepsilon)|^p
&\le
2^{p-1}
|f(x+2\varepsilon)
-f(x+2\varepsilon-4\varepsilon_je_j)|^p
\\
&\quad+
2^{p-1}
|f(x-2\varepsilon)
-f(x-2\varepsilon+4\varepsilon_je_j)|^p.
\end{aligned}
$$
Averaging over $(x,\varepsilon)$ and using translation invariance gives
$$
\mathbb E_x\|\partial_jh_x\|_p^p
\le
2^p
\mathbb E_x|f(x+4e_j)-f(x)|^p.
$$
Since
$$
f(x+4e_j)-f(x)
=
\sum_{\ell=1}^4
\bigl(
f(x+\ell e_j)-f(x+(\ell-1)e_j)
\bigr),
$$
we obtain
$$
|f(x+4e_j)-f(x)|^p
\le
4^{p-1}
\sum_{\ell=1}^4
|f(x+\ell e_j)-f(x+(\ell-1)e_j)|^p.
$$
After averaging in $x$,
$$
\mathbb E_x\|\partial_jh_x\|_p^p
\le
8^p
\mathbb E_x|f(x+e_j)-f(x)|^p.
$$
Therefore
$$
\sum_{j=1}^n
\mathbb E_x\|\partial_jh_x\|_p^p
\le
8^p
\sum_{j=1}^n
\mathbb E_x|f(x+e_j)-f(x)|^p.
$$

Similarly,
$$
h_x(\varepsilon)
=
f(x+2\varepsilon)-f(x-2\varepsilon),
$$
and a four-step telescoping argument gives
$$
|h_x(\varepsilon)|^p
\le
4^{p-1}
\sum_{\ell=-1}^{2}
|f(x+\ell\varepsilon)
-f(x+(\ell-1)\varepsilon)|^p.
$$
Averaging and using translation invariance yields
$$
\mathbb E_x\|h_x\|_p^p
\le
4^p
\mathbb E_{x,\varepsilon}
|f(x+\varepsilon)-f(x)|^p.
$$

Substituting the last two estimates into the previous inequality and
taking $p$th roots proves the lemma.
\end{proof}

Now we prove the metric $X_p$ inequality.

\begin{theorem}[Sharp metric $X_p$ inequality]
\label{thm:metric}
Suppose that $k,m,n\in\mathbb N$, $k\in[n]$, $p\ge2$, and
$m\ge\sqrt{n/k}$. Then every
$f:\mathbb Z_{8m}^n\to L^p$ satisfies
$$
\begin{aligned}
&\left(
\frac1{\binom nk}
\sum_{\substack{S\subseteq[n]\\|S|=k}}
\mathbb E_{x,\varepsilon}
\left[
\left\|
f(x+4m\varepsilon_S)-f(x)
\right\|_{L^p}^p
\right]
\right)^{1/p}
\\
&\qquad\lesssim
\frac{p}{\log p}\,m
\left(
\frac{k}{n}
\sum_{j=1}^n
\mathbb E_x
\left[
\|f(x+e_j)-f(x)\|_{L^p}^p
\right]
+
\left(\frac{k}{n}\right)^{p/2}
\mathbb E_{x,\varepsilon}
\left[
\|f(x+\varepsilon)-f(x)\|_{L^p}^p
\right]
\right)^{1/p}.
\end{aligned}
$$
Moreover, the order $p/\log p$ is sharp up to a universal
multiplicative constant.
\end{theorem}

Taking $M=4m$ in the definition of a metric $X_p$ space, the preceding
theorem shows that $L^p$ is a metric $X_p$ space with constant of order
at most $p/\log p$, up to a universal multiplicative constant.

\begin{proof}
As explained above, it suffices to prove the scalar-valued case. Fix
$S\subseteq[n]$ with $|S|=k$. By the triangle inequality,
$$
\begin{aligned}
|f(x+4m\varepsilon_S)-f(x)|
&\le
\left|
T_{[n]\setminus S}f(x+4m\varepsilon_S)
-
T_{[n]\setminus S}f(x)
\right|
\\
&\quad+
|f(x)-T_{[n]\setminus S}f(x)|
\\
&\quad+
|f(x+4m\varepsilon_S)
-
T_{[n]\setminus S}f(x+4m\varepsilon_S)|.
\end{aligned}
$$
Taking the $L^p$ norm with respect to $(x,\varepsilon)$ and then
averaging over all $S$ with $|S|=k$, Lemma~\ref{lem:metric-main}
controls the first term, while Lemma~\ref{lem:smoothing} and
translation invariance control the last two terms. Thus
$$
\begin{aligned}
&\left(
\frac1{\binom nk}
\sum_{\substack{S\subseteq[n]\\|S|=k}}
\mathbb E_{x,\varepsilon}
|f(x+4m\varepsilon_S)-f(x)|^p
\right)^{1/p}
\\
&\qquad\lesssim
\frac{p}{\log p}\,m
\left(
\frac{k}{n}
\sum_{j=1}^n
\mathbb E_x|f(x+e_j)-f(x)|^p
+
\left(\frac{k}{n}\right)^{p/2}
\mathbb E_{x,\varepsilon}
|f(x+\varepsilon)-f(x)|^p
\right)^{1/p}
\\
&\qquad\quad+
\left(
\mathbb E_{x,\varepsilon}
|f(x+\varepsilon)-f(x)|^p
\right)^{1/p}.
\end{aligned}
$$
Since $m\ge\sqrt{n/k}$, we have
$$
\left(
\mathbb E_{x,\varepsilon}
|f(x+\varepsilon)-f(x)|^p
\right)^{1/p}
\le
m
\left(\frac{k}{n}\right)^{1/2}
\left(
\mathbb E_{x,\varepsilon}
|f(x+\varepsilon)-f(x)|^p
\right)^{1/p}.
$$
Thus the final term is absorbed into the second term on the right-hand
side. Since $p/\log p$ is bounded below by a positive universal
constant for $p\ge2$, this gives the desired scalar estimate.

It remains to pass from the scalar-valued estimate to the
$L^p$-valued one. Write $L^p=L^p(\Omega,\mu)$ and let
$$
f:\mathbb Z_{8m}^n\to L^p(\Omega,\mu).
$$
For $\omega\in\Omega$, define
$$
f_\omega(x)=f(x)(\omega),
\qquad x\in\mathbb Z_{8m}^n.
$$
Let $C_1>0$ be a universal constant for which the scalar estimate above
holds. Applying that estimate to $f_\omega$, raising both sides to the
$p$th power, and then integrating with respect to $\omega$, we obtain,
by Fubini's theorem,
$$
\begin{aligned}
&\frac1{\binom nk}
\sum_{\substack{S\subseteq[n]\\|S|=k}}
\E_{x,\varepsilon}
\|f(x+4m\varepsilon_S)-f(x)\|_{L^p}^p
\\
&\qquad\le
\left(C_1\frac{p}{\log p}\right)^p m^p
\left[
\frac{k}{n}\sum_{j=1}^n
\E_x\|f(x+e_j)-f(x)\|_{L^p}^p
+
\left(\frac{k}{n}\right)^{p/2}
\E_{x,\varepsilon}
\|f(x+\varepsilon)-f(x)\|_{L^p}^p
\right].
\end{aligned}
$$
Taking $p$th roots gives the asserted $L^p$-valued inequality.

Finally, the factor $p/\log p$ cannot be improved up to a universal
multiplicative constant. This lower bound for the metric $X_p$ constant
of $L^p$ was established in the earlier theory of metric $X_p$
inequalities; see \cite{NaorSchechtman} and \cite[Remark~5]{Naor}.
\end{proof}

\noindent\textbf{Note added.}
After completion of this work, we became aware of the independent preprint of Yutong Zhang and Yaoran Yang \cite{ZhangYang}. Their manuscript was submitted to arXiv on August 14, 2026, but did not become publicly available until September 23, 2026. The present paper was submitted to arXiv on August 29, 2026. The two works were carried out independently and establish the same sharp $p/\log p$ dependence. Their methods are rather different from ours: Zhang and Yang proceed through a Walsh inverse-gradient decomposition, a dimension-free square-function estimate, and a fixed-cardinality Rosenthal comparison, whereas our argument is based directly on random-reveal martingales and Hitczenko's martingale Rosenthal inequality.

\section*{Acknowledgments}

The author acknowledges the use of AI tools. All mathematical arguments and proofs in the final manuscript were checked and written by the author.

\end{document}